\documentclass[11pt, reqno]{amsart}
\usepackage{indentfirst, amssymb, amsmath, amsthm, mathrsfs, setspace, indentfirst, enumerate,  mathrsfs, amsmath, amsthm}
\usepackage[bookmarksnumbered, colorlinks, plainpages]{hyperref}
\usepackage{mathrsfs}

\newtheorem*{exm1A}{Example 1.A}
\newtheorem*{exm1B}{Example 1.B}
\newtheorem*{exm1C}{Example 1.C}
\newtheorem*{exm1D}{Example 1.D}

\newtheorem*{ques1A}{Question 1.A}
\newtheorem*{ques1B}{Question 1.B}

\newtheorem*{theo1A}{Theorem 1.A}
\newtheorem*{theo1B}{Theorem 1.B}
\newtheorem*{theo1C}{Theorem 1.C}
\newtheorem*{theo1D}{Theorem 1.D}
\newtheorem*{theo1E}{Theorem 1.E}

\newtheorem*{rem1A}{Remark 1.A}

\newtheorem{theo}{Theorem}[section]
\newtheorem{lem}{Lemma}[section]

\newtheorem{exm}{Example}[section]

\newtheorem{rem}{Remark}[section]

\newcommand{\ol}{\overline}

\newcommand{\be}{\begin{equation}}
\newcommand{\ee}{\end{equation}}
\newcommand{\beas}{\begin{eqnarray*}}
\newcommand{\eeas}{\end{eqnarray*}}
\newcommand{\bea}{\begin{eqnarray}}
\newcommand{\eea}{\end{eqnarray}}
\newcommand{\lra}{\longrightarrow}
\numberwithin{equation}{section}

\numberwithin{equation}{section}

\begin{document}
\title[R\MakeLowercase {elationship between the value-sharing behavior of an entire function}.....]{\LARGE R\Large\MakeLowercase {elationship between the value-sharing behavior of an entire function and its derivative, and the analytic structure of a nonlinear differential equation}}
\date{}
\author[J. F. X\MakeLowercase{u}, S. M\MakeLowercase{ajumder} \MakeLowercase{and} L. M\MakeLowercase{ahato}]{J\MakeLowercase{unfeng} X\MakeLowercase{u}, S\MakeLowercase{ujoy} M\MakeLowercase{ajumder}$^{*}$ \MakeLowercase{and} L\MakeLowercase{ata} M\MakeLowercase{ahato}}

\address{Department of Mathematics, Wuyi University, Jiangmen 529020, Guangdong, People's Republic of China.}
\email{xujunf@gmail.com}

\address{Department of Mathematics, Raiganj University, Raiganj, West Bengal-733134, India.}
\email{sm05math@gmail.com}

\address{Department of Mathematics, Mahadevananda Mahavidyalaya, Monirampore Barrackpore, West Bengal-700120, India.}
\email{lata27math@gmail.com}

\renewcommand{\thefootnote}{}
\footnote{*\emph{Corresponding Author}: Sujoy Majumder.}
\renewcommand{\thefootnote}{\arabic{footnote}}
\setcounter{footnote}{0}

\begin{abstract} In this paper, we study uniqueness problems for entire functions that partially share two values with their higher-order derivatives. The results obtained here both improve and generalize the related results of Li and Yi \cite{LYi}, L\"{u} et al. \cite{LXY1} and Sauer and Schweizer \cite{SS1}. Furthermore, we show that our results reveal a deep relationship between the value-sharing behavior of an entire function $f$ and its $k$-th derivative $f^{(k)}$, and the analytic structure of a particular type of nonlinear differential equation. Several examples are provided to illustrate the necessity of the conditions used in our results.
\end{abstract}
\thanks{Typeset by \AmS -\LaTeX}
\maketitle

\section{{\bf Introduction}}
The study of entire functions that share values with their derivatives is a classical and active area in complex analysis, especially within Nevanlinna theory. When an entire function shares two values partially with its general derivatives, the behavior of the function becomes highly constrained. These constraints often lead to strong uniqueness theorems or explicit characterizations of the entire function.

Entire functions that share values with their derivatives arise naturally in the study of differential equations, normal families, and uniqueness theory. Investigating partial sharing (e.g., IM-sharing, CM-sharing, or sharing in the sense of reduced multiplicity) helps broaden classical uniqueness results while capturing more delicate analytical behavior.

Results of this type are especially useful for:\\
$\bullet$ establishing the solvability and uniqueness of nonlinear differential equations,\\
$\bullet$ understanding fixed points and periodic points of differential operators,\\
$\bullet$ studying functional equations involving $f$ and $f^{(k)}$,\\
$\bullet$ proving normality criteria for families of meromorphic functions.

\medskip
Let $f(z)$ and $g(z)$ be two non-constant meromorphic functions and $a$ be a finite complex number. If $g(z)-a=0$ whenever $f(z)-a=0$, we write $f(z)=a\Rightarrow g(z)=a$. If $f(z)=a\Rightarrow g(z)=a$ and $g(z)=a\Rightarrow f(z)=a$, we then write $f(z)=a\Leftrightarrow g(z)=a$ and we say that $f(z)$ and $g(z)$ share $a$ IM. If $f(z)-a$ and $g(z)-a$ have the same zeros with the same multiplicities, we write $f(z)=a\rightleftharpoons g(z)=a$ and we say that $f(z)$ and $g(z)$ share $a$ CM.\par

\medskip
In 1977, Rubel and Yang \cite{RY} were the first to investigate the uniqueness problem for an entire function $f(z)$ that shares two values with its first derivative $f^{(1)}(z)$. They established the following result.

\begin{theo1A}\cite{RY} Let $f(z)$ be a non-constant entire function and let $a$ and $b$ be two distinct finite complex numbers. If $f(z)=a\rightleftharpoons f^{(1)}(z)=a$ and $f(z)=b\rightleftharpoons f^{(1)}(z)=b$, then $f(z)\equiv f^{(1)}(z)$, i.e., $f(z)=A\exp(z)$, where $A\in\mathbb{C}\backslash \{0\}$.
\end{theo1A} 

\medskip
In 1979, Mues and Steinmetz \cite{Mues} further generalized Theorem 1.A by weakening the sharing condition from CM to IM, as stated in the following result.
\begin{theo1B}\cite[Satz 1]{Mues} Let $f(z)$ be a non-constant entire function and let $a$ and $b$ be two distinct finite complex numbers. If $f(z)=a\Leftrightarrow f^{(1)}(z)=a$ and $f(z)=b\Leftrightarrow f^{(1)}(z)=b$, then $f(z)\equiv f^{(1)}(z)$, i.e., $f(z)=A\exp(z)$, where $A\in\mathbb{C}\backslash \{0\}$.
\end{theo1B} 

These two results ushered in a new era in the study of uniqueness problems for entire and meromorphic functions that share two values with their derivatives, and they sparked a long-standing line of research in this field (see \cite{BM1}, \cite{LM1}-\cite{LY}, \cite{SM1}, \cite{SM2}, \cite{SMS}, \cite{MD1}, \cite{MRKS}, \cite{Mss}, \cite{MSS1}, \cite{MS1}, \cite{MSD1}, \cite{SS1}, \cite{XMD} and \cite[Chapter 8]{10}).

The following example demonstrates that the number of shared values in Theorems 1.A and 1.B cannot be reduced to one.
\begin{exm1A}\cite{B1} Let
\begin{align*}
f(z)=\exp(\exp(z))\int\limits_{0}^{z}\exp(-\exp(t))\left(1-\exp(t)\right)\;dt.
\end{align*}

It is easy to verify that 
\begin{align*}
f^{(1)}(z)-1=\exp(z)(f(z)-1)
\end{align*}
and so $f(z)=1\Leftrightarrow f^{(1)}(z)=1$, but $f(z)\not\equiv f^{(1)}(z)$.
\end{exm1A}

\smallskip
Essentially, Theorems 1.A and 1.B have been generalized in two main directions. The first direction replaces the sharing condition $f(z)=b\rightleftharpoons f^{(1)}(z)=b$ with the weaker requirement $f(z)=b\Rightarrow f^{(1)}(z)=b$
(see \cite{LYi,LXY1,LXC1,LY,Mss}). The second direction reverses the implication by replacing $f(z)=b\rightleftharpoons f^{(1)}(z)=b$
with $f^{(1)}(z)=b\Rightarrow f(z)=b.$

\smallskip
In this paper, we refer to the following result, due to Li and Yi \cite{LYi}, in which the original sharing condition $f(z)=b\rightleftharpoons f^{(1)}(z)=b$ is weakened using the concept of partial value sharing, namely $f^{(1)}(z)=b\Rightarrow f(z)=b$.

\begin{theo1C}\cite[Theorem 1]{LYi} Let $f(z)$ be a non-constant entire function and let $a$ and $b$ be two finite complex numbers such that $b\not=a,0$. If $f(z)=a\rightleftharpoons f^{(1)}(z)=a$ and $f^{(1)}(z)=b\Rightarrow f(z)=b$, then $f(z)\equiv f^{(1)}(z)$, i.e., $f(z)=A\exp(z)$, where $A\in\mathbb{C}\backslash \{0\}$.
\end{theo1C} 

\begin{rem1A} We observe that $b\not=a$ in Theorem 1.C. Therefore, if we assume $b\not=0$ in Theorem 1.A, then it is clear that Theorem 1.C generalizes Theorem 1.A through the notion of ``partial'' sharing of values.
\end{rem1A}
 
In the same paper, Li and Yi \cite{LYi} presented the following example to demonstrate the necessity of the condition ``$b\not=0$'' in Theorem 1.C.

\begin{exm1B}\cite{LYi}Let 
\begin{align*}
f(z)=c\exp\left(\frac{a}{c}z\right)+a-c,
\end{align*}
where $c$ is a non-zero constant. Note that $f^{(1)}(z)\not=0$. Clearly $f(z)=a\rightleftharpoons f^{(1)}(z)=a$ and $f^{(1)}(z)=0\Rightarrow f(z)=0$, but $f(z)\not\equiv f^{(1)}(z)$.
\end{exm1B}

\medskip
It would be a remarkable achievement to show that the conclusion of Theorem 1.C remains valid when the hypothesis $f(z)=a\rightleftharpoons f^{(1)}(z)=a$ is replaced by the weaker condition $f(z)=a\Leftrightarrow f^{(1)}(z)=a$.
In 2009, L\"{u}, Xu and Yi \cite{LXY1} accomplished this by establishing the following result. 

\begin{theo1D}\cite[Theorem 1.1]{LXY1} Let $f(z)$ be a non-constant entire function and let $a$ and $b$ be two non-zero finite complex numbers such that $b\not=a$. If $f(z)=a\Leftrightarrow f^{(1)}(z)=a$ and $f^{(1)}(z)=b\Rightarrow f(z)=b$, then $f(z)\equiv f^{(1)}(z)$, i.e., $f(z)=A\exp(z)$, where $A\in\mathbb{C}\backslash \{0\}$.
\end{theo1D}

\medskip
L\"{u} et al. \cite{LXY1} also presented the following examples to demonstrate the necessity of the condition ``$a$ and $b$ be two non-zero finite complex numbers'' in Theorem 1.D.

\begin{exm1C}\cite[Example 1]{LXY1}
Let $f(z)=\frac{b}{4}z^2$. Clearly $f(z)=0\Leftrightarrow f^{(1)}(z)=0$ and $f^{(1)}(z)=b\Rightarrow f(z)=b$, but $f(z)\not\equiv f^{(1)}(z)$.
\end{exm1C}

\begin{exm1D}\cite[Example 2]{LXY1} For the function $f(z)$ in Example 1.B, we have $f(z)=a\Leftrightarrow f^{(1)}(z)=a$ and $f^{(1)}(z)=b\Rightarrow f(z)=b$, but $f(z)\not\equiv f^{(1)}(z)$.
\end{exm1D}

At the end of their paper, L\"{u} et al. \cite{LXY1} raised the following questions for further investigation:
\begin{ques1A} What will happen if ``$b=0$'' in Theorem 1.D?
\end{ques1A}
\begin{ques1B} What will happen if the hypothesis ``$f(z)=a\Leftrightarrow f^{(1)}(z)=a$'' is replaced by ``$f(z)=a\Rightarrow  f^{(1)}(z)=a$'' in Theorem 1.D?
\end{ques1B}

\medskip
In 2024, Sauer and Schweizer \cite{SS1} considered {\bf Question 1.B} and resolved it completely using techniques from normal family theory and number theory. We recall their results below.

\begin{theo1E}\cite[Theorem 1.1]{SS1} Let $f(z)$ be a non-constant entire function and let $a$ and $b$ be two non-zero finite complex numbers such that $b\not=a$. If $f(z)=a\Rightarrow  f^{(1)}(z)=a$ and $f^{(1)}(z)=b\Rightarrow f(z)=b$, then $f(z)$ is one of the following: 
\begin{enumerate}  
\item[(1)] $f(z)=az+C$, where $C$ is any constant. These are all non-constant polynomial solutions,
\item[(2)] $f(z)\equiv f^{(1)}(z)$, i.e., $f(z)=C\exp(z)$, where $C$ is any non-zero constant,
\item[(3)] $f(z)=C\exp\left(\frac{b}{b-a}z\right)+a$, where $C$ is a non-zero constant. In this case $f(z)$ and $f^{(1)}(z)$ do even share the value $b$. Moreover, $a$ is a Picard value of $f(z)$,
\item[(4)]$f(z)=6aC \exp\left(\frac{1}{6} z\right)\left( C \exp\left(\frac{1}{6} z\right)-1\right)+a$, where $b=-\frac{a}{8}$ and $C$ is a non-zero constant.
\end{enumerate}
\end{theo1E}

\smallskip
Let $f(z)$ and $g(z)$ be two non-constant meromorphic functions and $a$ be a finite complex number. Let $z_n$ be zeros of $f(z)-a$ with multiplicity $h(n)$. If $z_n$ are also zeros of $g(z)-a$ of multiplicity $h(n)$ at least, then we write $f(z)=a\mapsto g(z)=a$.  
Clearly $f(z)=a\mapsto g(z)=a$ always implies that $f(z)=a\Rightarrow g(z)=a$, but not conversely. But if $f(z)=a\Rightarrow f^{(1)}(z)=a$ holds, then $f(z)=a\mapsto f^{(1)}(z)=a$ always holds.

Throughout the paper, we denote by $\mu(f)$, $\rho(f)$ and $\lambda(f)$ the lower order of $f$, the order of $f$ and the exponent of convergence of zeros of $f$ respectively. 
If $\mu(f)=\rho(f)$, we say that $f$ is of regular growth. Also if
\begin{align*}
\limsup\limits_{r\to \infty }\frac{\log^+ N\left(r,a;f\right)}{\log r}<\rho(f)
\end{align*}
for $\rho(f)>0$, then $a\in\mathbb{C}\cup\{\infty\}$ is said to be a Borel exceptional value of $f$. Obviously when $\rho>0$, Picard exceptional values are Borel's.

\section{{\bf Main results}}
One of the main objectives of this paper is to provide an affirmative answer to {\bf Question 1.A}. It is also natural to ask whether the conclusions of Theorems 1.C-1.E remain valid when the first derivative $f^{(1)}(z)$ in these theorems is replaced by the higher-order derivative $f^{(k)}(z)$. In this paper, we address this question by establishing the following results.

\subsection{{\bf When $f(z)=a\Leftrightarrow f^{(k)}(z)=a$ and $f^{(k)}(z)=b\Rightarrow f(z)=b$}}

\begin{theo}\label{t1.1} Let $f(z)$ be a non-constant entire function such that all the zeros of $f(z)-a$ have multiplicity at least $k$, where $k$ is a positive integer and $a$ is a non-zero constant. If $f(z)=a\mapsto f^{(k)}(z)=a$, $f^{(k)}(z)=a\Rightarrow f(z)=a$ and $f^{(k)}(z)=b\Rightarrow f(z)=b$, where $b(\neq 0,a)$ is a constant, then only one of the following cases holds: 
\begin{enumerate}
\item[(1)] $k=1$ and $f(z)=A\exp(z)$, where $A$ is a non-zero constant,
\item[(2)] $k=2$ and $f(z)=c_0\exp(z)+c_1\exp(-z)$, where $c_0$ and $c_1$ are non-zero constants such that $c_0c_1=\frac{a^2}{4}$.
\end{enumerate}
\end{theo}

\begin{theo}\label{t1.2} Let $f(z)$ be a non-constant entire function such that all the zeros of $f(z)-a$ have multiplicity at least $k$, where $k$ is a positive integer and $a$ is a non-zero constant. If $f(z)=a\rightleftharpoons f^{(k)}(z)=a$ and $f^{(k)}(z)=0\Rightarrow f(z)=0$, then only one of the following cases holds: 
\begin{enumerate}
\item[(1)] $k=1$ and $f(z)=A\exp(z)$, where $A$ is a non-zero constant,
\item[(2)] $k=2$ and $f(z)=c_0\exp(z)+c_1\exp(-z)$, where $c_0$ and $c_1$ are non-zero constants such that $c_0c_1=\frac{a^2}{4}$,
\item[(3)] $f(z)=\frac{a(\beta-1)}{\beta}+\frac{C}{\beta}\exp(\lambda z)$,
where $C$, $\beta(\neq 1)$ and $\lambda$ are non-zero constants such that $\lambda^k=\beta$.
\end{enumerate}
\end{theo}

We now make the following observations regarding Theorems \ref{t1.1} and \ref{t1.2}.
\begin{enumerate}
\item[(i)] We have improved as well as generalized Theorems 1.C and 1.D by considering the general derivative $f^{(k)}$ instead of only the first derivative $f^{(1)}$. Moreover, we explore all possible forms of the function $f$.
\item[(ii)] We have completely solved {\bf Question 1.A} in Theorem \ref{t1.2} in the case where $f$ and $f^{(k)}$ share $a$ CM.
\item[(iii)] The condition that all the zeros of $f(z)-a$ have multiplicity at least $k$ is necessary in Theorem \ref{t1.1}. For example, consider
\begin{align*}
f(z)=\exp(z)+\exp(-z),
\end{align*}
where $a=4$, $b=1$ and $k=2$. Note that all the zeros of $f(z)-a$ are simple and $c_0c_1=1\neq\frac{a^2}{4}=4$. Obviously $f(z)=a\mapsto f^{(k)}(z)=a$, $f^{(k)}(z)=a\Rightarrow f(z)=a$ and $f^{(k)}(z)=b\Rightarrow f(z)=b$.
However, $f(z)$ does not satisfy any of the cases of Theorem \ref{t1.1}. This shows that the multiplicity condition cannot be omitted.
\item[(iv)] Suppose $k=1$. If $f(z)=a\mapsto f^{(k)}(z)=a$, then the function $f(z)-a$ has only simple zeros.
Hence, the sharing conditions $f(z)=a\mapsto f^{(k)}(z)=a$ and $f^{(k)}(z)=a\Rightarrow f(z)=a$ are equivalent to $f(z)=a\Leftrightarrow f^{(z)}(z)=a$. Consequently, Theorem 1.D follows as a direct corollary of Theorem \ref{t1.1}.
\end{enumerate}
\medskip

\subsection{{\bf When $f(z)=a\Rightarrow f^{(k)}(z)=a$ and $f^{(k)}(z)=b\Rightarrow f(z)=b$}}

\begin{theo}\label{t1.3} Let $f(z)$ be a non-constant entire function such that either $f(z)-a$ has no zeros or all the zeros of $f(z)-a$ have multiplicity at least $k$, where $k$ is a positive integer and $a$ is a non-zero constant. If $f(z)=a\mapsto f^{(k)}(z)=a$ and $f^{(k)}(z)=b\Rightarrow f(z)=b$, where $b(\neq 0,a)$ is a constant, then only one of the following cases holds: 
\begin{enumerate}
\item[(1)] $k=1$ and $f(z)=A\exp(z)$, where $A$ is a non-zero constant,
\item[(2)] $k=2$ and $f(z)=c_0\exp(z)+c_1\exp(-z)$, where $c_0$ and $c_1$ are non-zero constants such that $c_0c_1=\frac{a^2}{4}$,
\item[(3)] $f(z)=\frac{a}{k!}(z-z_1)^k+a$,
\item[(4)] $k=1$ and $f(z)=6aC \exp\left(\frac{1}{6} z\right)\left( C \exp\left(\frac{1}{6} z\right)-1\right)+a$, where $b=-\frac{a}{8}$ and $C$ is a non-zero constant,
\item[(5)] $f(z)=C\exp\left(\lambda z\right)+a$, where $C$ and $\lambda$ are non-zero constants such that $\lambda^k=\frac{b}{b-a}\neq 1$.
\end{enumerate}
\end{theo}

We now make the following observations regarding Theorem \ref{t1.3}.
\begin{enumerate}
\item[(i)] Suppose $k=1$. If $f(z)=a\mapsto f^{(k)}(z)=a$, then the function $f(z)-a$ has only simple zeros.
Hence, the sharing condition $f(z)=a\mapsto f^{(k)}(z)=a$ is equivalent to $f(z)=a\Rightarrow f^{(z)}(z)=a$. Consequently, Theorem 1.E follows as an immediate corollary of Theorem \ref{t1.3}. Thus, we have generalized Theorem 1.E by considering the higher-order derivative 
$f^{(k)}$ instead of only the first derivative $f^{(1)}$.
\item[(ii)] For $k=1$, the assumption that ``either $f(z)-a$ has no zeros or all the zeros of $f(z)-a$ have multiplicity at least $k$'' is unnecessary.
\end{enumerate}
\smallskip

In the following two theorems, we restrict our attention to the case $b=0$.

\begin{theo}\label{t1.4} Let $f(z)$ be a non-constant entire function such that all the zeros of $f(z)-a$ have multiplicity at least $k$, where $k$ is a positive integer and $a$ is a non-zero constant and let $\lambda(f)<\rho(f)$. If $f(z)=a\Rightarrow f^{(k)}(z)=a$ and $f^{(k)}(z)=0\mapsto f(z)=0$, then $k=1$ and $f(z)=A\exp(z)$, where $A$ is a non-zero constant.
\end{theo}

\begin{theo}\label{t1.5}  Let $f(z)$ be a non-constant entire function such that either $f(z)$ has no zeros or all the zeros of $f(z)$ have multiplicity at least $k$, where $k$ is a positive integer and let $\lambda(f)<\rho(f)$. If $f(z)=a\Rightarrow f^{(k)}(z)=a$ and $f^{(k)}(z)=0\Rightarrow f(z)=0$, where $a$ is a non-zero constant, then $f(z)=A\exp(\lambda z)$, where $A$ and $\lambda$ are non-zero constants such that $\lambda^k=1$.
\end{theo}

\begin{rem} For $k=1$, the assumption that ``either $f(z)$ has no zeros or all the zeros of $f(z)$ have multiplicity at least $k$'' in Theorem \ref{t1.5} is unnecessary.
\end{rem}

\begin{rem} The examples below demonstrate that the assumptions ``all the zeros of $f(z)-a$ have multiplicity at least $k$'' and ``$\lambda(f)<\rho(f)$'' in Theorems \ref{t1.4} are indeed sharp.
\end{rem}

\begin{exm} Let
\begin{align*}
f(z)=\exp(z^2)-1.
\end{align*}
If we take $a=-1$, then $f(z)-a$ has no zeros. Clearly $f(z)=a\Rightarrow f^{(1)}(z)=a$ and $f^{(1)}(z)=0\mapsto f(z)=0$, but $f(z)$ does not satisfy the conclusion of Theorem \ref{t1.4}.
\end{exm}
\begin{exm} Let 
\begin{align*}
f(z)=\frac{d}{c}\exp(c z)+a-\frac{a}{c},
\end{align*}
where $a$, $c$ and $d$ are non-zero constants such that $c\neq 1$. Clearly $\lambda(f)=\rho(f)=1$ and $f(z)-a$ has only simple zeros. On the other hand, we observe that $f(z)=a\Rightarrow f^{(1)}(z)=a$ and $f^{(1)}(z)=0\mapsto f(z)=0$. However, $f(z)$ does not satisfy the conclusion of Theorem \ref{t1.4}.
\end{exm}

\begin{rem} The examples below demonstrate that the requirements ``either $f(z)$ has no zeros or all the zeros of $f(z)$ have multiplicity at least $k$'' and ``$\lambda(f)<\rho(f)$'' in Theorems \ref{t1.5} are indeed sharp.
\end{rem}

\begin{exm} Let 
\begin{align*}
f(z)=C\exp(\sqrt{2} z)+1,
\end{align*}
where $a=2$ and $k=2$. Clearly $\lambda(f)=\rho(f)=1$ and $f(z)$ has only simple zeros. Note that $f^{(2)}(z)$ has no zeros whereas $f(z)$ has infinitely many zeros and and hence $f^{(2)}(z)=0\Rightarrow f(z)=0$. Moreover, we observe that $f(z)=a\Rightarrow f^{(2)}(z)=a$. However $f(z)$ does not satisfy the conclusion of Theorem \ref{t1.5}.
\end{exm}

\begin{exm} Let 
\begin{align*}
f(z)=\frac{a}{k!}(z-z_1)^k+a,
\end{align*}
where $a$ and $z_1$ are constants such that $a\neq 0$. Clearly $\lambda(f)=\rho(f)$, $f(z)$ has only simple zeros and $f^{(k)}(z)=a\neq 0$. Therefore $f^{(k)}(z)$ has no zeros and hence $f^{(k)}(z)=0\Rightarrow f(z)=0$. Moreover $f(z)=a\Rightarrow f^{(k)}(z)=a$. However $f(z)$ does not satisfy the conclusion of Theorem \ref{t1.5}.
\end{exm}

\section {{\bf Auxiliary lemmas}} 

Let $f(z)$ be a meromorphic function in a domain $\Omega\subset\mathbb{C}$. Then the derivative of $f(z)$ at $z_{0}\in\Omega$ in the spherical metric, called the spherical derivative, is denoted by $f^{\#}(z_{0})$, where
\beas  f^{\#}(z_{0})=\left\{\begin{array}{clcr} & \frac{|f^{(1)}(z_{0})|}{1+|f(z_{0})|^{2}},&\;\;\;\;\;\;\;\text{if}\;\;z_{0}\in\Omega\;\text{is\;not\;a\;pole} \\   
& \lim\limits_{z\rightarrow z_{0}}\frac{|f^{(1)}(z)|}{1+|f(z)|^{2}},&\;\;\text{if}\;\;z_{0}\in\Omega\;\text{is\;a\;pole}.\end{array}\right.\eeas 

A family $\mathcal{F}$ of functions meromorphic in domain $\Omega\subset \mathbb{C}$ is said to be normal in $\Omega$ if every sequence $\{f_{n}\}_{n}\subseteq \mathcal{F}$ contains a subsequence which converges spherically uniformly on compact subsets of $\Omega$ (see \cite{JS1}).

\smallskip
It is assumed that the reader is familiar with the following well known Marty Criterion.

\begin{lem}\label{L.1}\cite{JS1} A family $\mathcal{F}$ of meromorphic functions on a domain $\Omega$ is normal if and only if for each compact subset $K\subseteq \Omega$, there exists $M\in\mathbb{R}^+$ such that $f^{\#}(z)\leq M$ $\forall$ $f\in\mathcal{F}$ and $z\in K$. 
\end{lem} 

\smallskip
Now we recall the following well known Zalcman's lemma.

\begin{lem}\label{L.2}\cite{LZ1} Let $\mathcal{F}$ be a family of meromorphic functions in the unit disc $\Delta$ and $\alpha$ be a real number satisfying $-1<\alpha<1$. Then $\mathcal{F}$ is not normal at a point $z_{0}\in \Delta$, if and only if there exist
\begin{enumerate}\item[(i)] points $z_{n}\in \Delta$, $z_{n}\rightarrow z_{0}$,\item[(ii)] positive numbers $\rho_{n}$, $\rho_{n}\rightarrow 0^{+}$ and \item[(iii)] functions $f_{n}\in \mathcal{F}$,
\end{enumerate}
such that $\rho_{n}^{\alpha} f_{n}(z_{n}+\rho_{n} \zeta)\rightarrow g(\zeta)$ spherically uniformly on compact subsets of $\mathbb{C}$, where $g$ is a non-constant meromorphic function. The function $g$ may be taken to satisfy the normalisation $g^{\#}(\zeta)\leq g^{\#}(0)=1 (\zeta \in \mathbb{C})$.
\end{lem}

The following lemma is due to Chang and Zalcman \cite{CZ}.
\begin{lem}\label{L.3}\cite[Lemma 2]{CZ} Let $f(z)$ be a non-constant meromorphic function such that $N(r,f)=O(\log r)$. If $f(z)$ has bounded spherical derivative on $\mathbb{C}$, then $\rho(f)\leq 1$.
\end{lem}

\begin{lem}\label{L.4} Let $\mathcal{F}$ be a family of holomorphic functions in a domain $\Omega$, $k$ be a positive integer and let $a(\neq 0)$ and  $b$ be two finite complex numbers such that $a\neq b$. Then $\mathcal{F}$ is normal in $\Omega$, if for each $f\in\mathcal{F}$, we have 
\begin{enumerate}
\item[(1)] all the zeros of $f(z)-a$ have multiplicity at least $k$,
\item[(2)] $f(z)=a\Rightarrow f^{(k)}(z)=a$,
\item[(3)] $f^{(k)}(z)=b\Rightarrow f(z)=b$. 
\end{enumerate}
\end{lem}
\begin{proof} We know that normality is a local property. Therefore it is enough to show that $\mathcal{F}$ is normal at each point $z_0\in \Omega$. Without loss of generality, we assume that $\Omega=\Delta=\{z: |z|<1\}$. Suppose on the contrary that $\mathcal{F}$ is not normal in $\Delta$. Consequently we may assume that $\mathcal{F}$ is not normal at $z_0\in\Delta$. 
Then by \textrm{Lemma \ref{L.2}}, there exist
\begin{enumerate} 
\item[(i)] points $z_{n}\in\Delta$, $z_{n}\rightarrow z_0$, 
\item[(ii)] positive numbers $\rho_{n}$, $\rho_{n}\rightarrow 0$, 
\item[(iii)] functions $f_{n}\in \mathcal{F}$
 \end{enumerate}
such that
\begin{align}
\label{L4.1} F_{n}(\zeta)=f_{n}(z_{n}+\rho_{n}\zeta)\rightarrow F(\zeta)
\end{align}  
spherically uniformly on compact subsets of $\mathbb{C}$, where $F(\zeta)$ is a non-constant entire function such that $F^{\#}(\zeta)\leq F^{\#}(0)=1\;(\zeta\in\mathbb{C})$. Since $F^{\#}(\zeta)\leq 1$, $\forall$ $\zeta\in\mathbb{C}$, by Lemma \ref{L.3}, we have  $\rho(F)\leq 1$. Since $f_{n}(z_{n}+\rho_{n}\zeta)-a\rightarrow F(\zeta)-a$, by Hurwitz's theorem, we conclude that all the zeros of $F(\zeta)-a$ have multiplicity at least $k$. Clearly, $F^{(k)}(\zeta)\not\equiv 0$; otherwise $F(\zeta)$ would be a polynomial of degree less than $k$, and therefore $F(\zeta)-a$ could not have a zero of multiplicity at least $k$. Now from (\ref{L4.1}), we have
\begin{align}
\label{L4.2} F_n^{(k)}(\zeta)=\rho_n^{k}f_n^{(k)}(z_n+\rho_n \zeta)\rightarrow F^{(k)}(\zeta)
\end{align}
spherically uniformly on compact subsets of $\mathbb{C}$.

\medskip
First we prove that $F(\zeta)=a\Rightarrow F^{(k)}(\zeta)=0$. Let $F(\zeta_0)=a$. Then by (\ref{L4.1}) and  Hurwitz's theorem there exists a sequence $\{\zeta_n\}$ such that $\zeta_n\rightarrow \zeta_0$ and $f_{n}(z_{n}+\rho_{n}\zeta_n)=a$ for sufficiently large values of $n$. This implies that $f_n(z_n+\rho_n \zeta_n)=a$. Since $f(z)=a\Rightarrow f^{(k)}(z)=a$, we have $f^{(k)}_n(z_n+\rho_n \zeta_n)=a$. Consequently, from (\ref{L4.2}), we have
\begin{align*}
F^{(k)}(\zeta_0)=\lim\limits_{n\rightarrow \infty}F^{(k)}_n(\zeta_n)=\lim\limits_{n\rightarrow \infty} \rho_n^{k} a=0.
\end{align*}

This shows that $F(\zeta)=a\Rightarrow F^{(k)}(\zeta)=0$.

\medskip
Next we prove that $F^{(k)}(\zeta)=0\Rightarrow F(\zeta)=b$ in $\mathbb{C}$. Clearly from (\ref{L4.2}), we have 
\begin{align}
\label{L4.3} F_n^{(k)}(\zeta)=\rho_n^{k}\left\lbrace f_n^{(k)}(z_n+\rho_n \zeta)-b\right\rbrace \rightarrow F^{(k)}(\zeta)
\end{align}
spherically uniformly on compact subsets of $\mathbb{C}$. Let $F^{(k)}(\zeta_0)=0$. Now by (\ref{L4.3}) and Hurwitz's theorem there exists a sequence $\{\zeta_n\}$ such that $\zeta_n\rightarrow \zeta_0$ and $f_n^{(k)}(z_n+\rho_n \zeta_n)=b$. Since $f^{(k)}(z)=a\Rightarrow f(z)=a$, we have $f_n(z_n+\rho_n \zeta_n)=b$ and so from (\ref{L4.1}), we get
\begin{align*}
 F(\zeta_0)=\lim\limits_{n\rightarrow \infty}F_n(\zeta_n)=\lim\limits_{n\rightarrow \infty} f_n(z_n+\rho_n \zeta_n)=b.
\end{align*}  

This shows that $F^{(k)}(\zeta)=0\Rightarrow F(\zeta)=b$. Finally, we have $F(\zeta)=a\Rightarrow F^{(k)}(\zeta)=0$ and $F^{(k)}(\zeta)=0\Rightarrow F(\zeta)=b$. Since $a\not=b$, we get a contradiction. Thus, all of the above considerations imply that $\mathcal{F}$ is normal at $z_0$. Consequently $\mathcal{F}$ is normal in $\Omega$.
\end{proof}

\begin{lem}\label{L.5} Let $\mathcal{F}$ be a family of holomorphic functions in a domain $\Omega$, $a$ and $b$ be two non-zero finite complex numbers such that $a\neq b$ and let $k$ be a positive integer. Then $\mathcal{F}$ is normal in $\Omega$, if for each $f\in\mathcal{F}$, we have   
\begin{enumerate}
\item[(1)] either $f(z)-a$ has no zeros or all the zeros of $f(z)-a$ have multiplicity at least $k$,
\item[(2)] $f(z)=a\Rightarrow f^{(k)}(z)=a$,
\item[(3)] $f^{(k)}(z)=b\Rightarrow f(z)=b$.
\end{enumerate}
\end{lem}

\begin{proof} We know that normality is a local property. Therefore it is enough to show that $\mathcal{F}$ is normal at each point $z_0\in \Omega$. Without loss of generality, we assume that $\Omega=\Delta=\{z: |z|<1\}$. Suppose on the contrary that $\mathcal{F}$ is not normal in $\Delta$. Set $\mathcal{F}_a=\{f-a: f\in\mathcal{F}\}$. Then $\mathcal{F}_a$ is not normal in $\Delta$. Consequently, we may assume that $\mathcal{F}_a$ is not normal at $z_0\in\Delta$. 
Then by Lemma \ref{L.2}, there exist
\begin{enumerate} 
\item[(i)] points $z_{n}\in\Delta$, $z_{n}\rightarrow z_0$, 
\item[(ii)] positive numbers $\rho_{n}$, $\rho_{n}\rightarrow 0$, 
\item[(iii)] functions $f_{n}\in \mathcal{F}_a$
 \end{enumerate}
such that
\begin{align}
\label{L5.1} F_{n}(\zeta)=\rho_{n}^{-\frac{1}{2}}\lbrace f_{n}(z_{n}+\rho_{n}\zeta)-a\rbrace\rightarrow F(\zeta)
\end{align}  
spherically uniformly on compact subsets of $\mathbb{C}$, where $F(\zeta)$ is a non-constant entire function such that $F^{\#}(\zeta)\leq F^{\#}(0)=1\;(\zeta\in\mathbb{C})$. Since $F^{\#}(\zeta)\leq 1$, $\forall$ $\zeta\in\mathbb{C}$, by Lemma \ref{L.1}, we get $\rho(F)\leq 1$. On the other hand from the proof of Zalcman's lemma (see \cite{LZ, LZ1}), we get
\begin{align}
\label{L5.2} \rho_{n}=\frac{1}{g_{n}^{\#}(z_{n})},
\end{align}
where $g_n(z_n)=f_{n}(z_{n})-a$. Also by Hurwitz's theorem, we conclude that either $F(\zeta)$ has no zeros or all the
zeros of $F(\zeta)$ have multiplicity at least $k$. Now from (\ref{L5.1}), we have
\bea\label{L5.3} F_n^{(k)}(\zeta)=\rho_n^{k-\frac{1}{2}}f_n^{(k)}(z_n+\rho_n \zeta)\rightarrow F^{(k)}(\zeta)\eea
spherically uniformly on compact subsets of $\mathbb{C}$. 
We consider the following two cases.\par
\smallskip

{\bf Case 1.} Suppose $F(\zeta)$ has no zeros. Then by Hadamard's factorization theorem, we may assume that 
\begin{align}
\label{L5.6} F(\zeta)=A\exp(\lambda \zeta),
\end{align}
where $A$ and $\lambda$ are non-zero constants. Clearly from (\ref{L5.1}) and (\ref{L5.3}), we have 
\begin{align}
\label{L5.7} \frac{F_{n}^{(1)}(\zeta)}{F_{n}(\zeta)}=\rho_{n}\frac{f_{n}^{(1)}(z_{n}+\rho_{n}\zeta)}{f_{n}(z_{n}+\rho_{n}\zeta)-a}\rightarrow \frac{F^{(1)}(\zeta)}{F(\zeta)}=\lambda,
\end{align}
spherically uniformly on compact subsets of $\mathbb{C}$. Now from (\ref{L5.2}) and (\ref{L5.7}), we get 
\begin{align*}
 \rho_{n}\left|\frac{f_{n}^{(1)}(z_{n})}{f_{n}(z_{n})-a}\right|=\frac{1+|f_{n}(z_{n})-a|^{2}}{|f_{n}^{(1)}(z_{n})|}\frac{|f_{n}^{(1)}(z_{n})|}{|f_{n}(z_{n})-a|}
=\frac{1+|f_{n}(z_{n})-a|^{2}}{|f_{n}(z_{n})-a|}\rightarrow \left|\frac{F^{(1)}(0)}{F(0)}\right|=|\lambda|,
\end{align*}
which implies that 
\begin{align*}
\lim\limits_{n\rightarrow \infty} \lbrace f_{n}(z_{n})-a\rbrace\not=0,\infty
\end{align*}
and so from (\ref{L5.1}), we have
\begin{align*}
F_{n}(0)=\rho_{n}^{-\frac{1}{2}}\lbrace f_{n}(z_{n})-a\rbrace \rightarrow \infty.
\end{align*}

But, from (\ref{L5.1}) and (\ref{L5.6}), we have $F_{n}(0)\rightarrow F(0)=A$. So, we get a contradiction.\par 
\smallskip

{\bf Case 2.} Suppose all the zeros of $F(\zeta)$ have multiplicity at least $k$. Clearly, $F^{(k)}(\zeta)\not\equiv 0$; otherwise $F(\zeta)$ would be a polynomial of degree less than $k$, and therefore could not have a zero of multiplicity at least $k$.\par

\medskip
First we prove that $F(\zeta)=0\Rightarrow F^{(k)}(\zeta)=0$. Let $F(\zeta_0)=0$. Then by Hurwitz's theorem there exists a sequence $\{\zeta_n\}$ such that $\zeta_n\rightarrow \zeta_0$ and
\begin{align*}
F_{n}(\zeta_n)=\rho_{n}^{-\frac{1}{2}}\lbrace f_{n}(z_{n}+\rho_{n}\zeta_n)-a\rbrace=0
\end{align*}
for sufficiently large values of $n$.
This implies that $f_n(z_n+\rho_n \zeta_n)=a$. Since $f(z)=a\Rightarrow f^{(k)}(z)=a$, we have $f^{(k)}_n(z_n+\rho_n \zeta_n)=a$. Consequently from (\ref{L5.3}), we have
\begin{align*}
F^{(k)}(\zeta_0)=\lim\limits_{n\rightarrow \infty}F^{(k)}_n(\zeta_n)=\lim\limits_{n\rightarrow \infty} \rho_n^{k-\frac{1}{2}} a=0.
\end{align*}

This shows that $F(\zeta)=0\Rightarrow F^{(k)}(\zeta)=0$.

\medskip
Next we prove that $F^{(k)}(\zeta)\neq 0$ in $\mathbb{C}$. If not, suppose $F^{(k)}(\zeta_0)=0$. Since $F_{n}^{(k)}(\zeta)\rightarrow F^{(k)}(\zeta)$ as $n\to\infty$ and $F^{(k)}(\zeta_0)=0$, by (\ref{L5.3}) and Hurwitz's theorem there exists a sequence $\{\zeta_n\}$ such that $\zeta_n\rightarrow \zeta_0$ and
\begin{align*}
 F_{n}^{(k)}(\zeta_n)=\rho_{n}^{k-\frac{1}{2}}\lbrace f_{n}^{(k)}(z_{n}+\rho_{n}\zeta_n)-b\rbrace=0
 \end{align*}
for sufficiently large values of $n$. This implies that $f_n^{(k)}(z_n+\rho_n \zeta_n)=b$. Since $f^{(k)}(z)=b\Rightarrow f(z)=b$, we have $f_n(z_n+\rho_n \zeta_n)=b$ and so from (\ref{L5.1}), we get
\begin{align}
\label{L5.5} F_{n}(\zeta_n)=\rho_{n}^{-\frac{1}{2}}\lbrace f_{n}(z_{n}+\rho_{n}\zeta_n)-a\rbrace=\rho_{n}^{-\frac{1}{2}}\lbrace b-a\rbrace=(b-a)\rho_{n}^{-\frac{1}{2}}.
\end{align} 

Since $a\neq b$ and $\rho_n\to 0$, from (\ref{L5.5}), we get
\begin{align*}
F(\zeta_0)=\lim\limits_{n\rightarrow \infty}F_n(\zeta_n)=\infty,
\end{align*}
which contradicts the fact that $F^{(k)}(\zeta_0)=0$. Hence $F^{(k)}(\zeta)\neq 0$ in $\mathbb{C}$.

Finally, we have $F(\zeta)=0\Rightarrow F^{(k)}(\zeta)=0$ and $F^{(k)}(\zeta)\neq 0$ in $\mathbb{C}$. Consequently $F(\zeta)\neq 0$ in $\mathbb{C}$. Now proceeding in the same way as done in Case 1, we get a contradiction.

Thus, all of the above considerations imply that $\mathcal{F}$ is normal at $z_0$. Hence $\mathcal{F}$ is normal in $\Omega$. 
\end{proof}

\begin{lem}\label{L.6} Let $\mathcal{F}$ be a family of holomorphic functions in a domain $\Omega$, $k$ be a positive integer and let $a$ be a non-zero constant. Then $\mathcal{F}$ is normal in $\Omega$, if for each $f\in\mathcal{F}$, we have 
\begin{enumerate}
\item[(1)] all the zeros of $f(z)-a$ have multiplicity at least $k$,
\item[(2)] $f(z)=a\Rightarrow f^{(k)}(z)=a$,
\item[(3)] $f^{(k)}(z)=0\Rightarrow f(z)=0$. 
\end{enumerate}
\end{lem}

\begin{proof}  
The proof of the lemma follows directly from that of Lemma \ref{L.5}, and we therefore omit the details.
\end{proof}

\begin{lem}\label{L.7} Let $\mathcal{F}$ be a family of holomorphic functions in a domain $\Omega$, $a$ be a non-zero constant and let $k$ be a positive integer. Then $\mathcal{F}$ is normal in $\Omega$, if for each $f\in\mathcal{F}$, we have  
\begin{enumerate}
\item[(1)] either $f(z)$ has no zeros or all the zeros of $f(z)$ have multiplicity at least $k$,
\item[(2)] $f(z)=a\Rightarrow f^{(k)}(z)=a$,
\item[(3)] $f^{(k)}(z)=0\Rightarrow f(z)=0$. 
\end{enumerate}
\end{lem}
\begin{proof} We prove Lemma \ref{L.7} by considering the following two cases.\par

{\bf Case 1.} Let $k=1$. The proof follows directly from the proof of Lemma \ref{L.5}, and we therefore omit the details.\par

{\bf Case 2.} Let $k\geq 2$. In this case, proceeding in the same way as done in the proof of Lemma \ref{L.4}, we arrive at (\ref{L4.1}), 
where $F(\zeta)$ is a non-constant entire function such that $\rho(F)\leq 1$ and either $F(\zeta)$ has no zeros or all the zeros of $F(\zeta)$ have multiplicity at least $k$. 

We claim that $a$ is not a Picard exceptional of $F$. If not, suppose $a$ is a Picard exceptional value of $F$. Now using Second fundamental theorem, we have 
\begin{align*}
T(r, F)\leq \bar N(r, 0;F)+\bar N(r, a;F)+S(r, F)\leq \frac{1}{k}N(r, 0;F)+S(r,F)\leq \frac{1}{k}T(r,F)+S(r,F),
\end{align*}
which is impossible since $k\geq 2$. Hence $a$ is not a Picard exceptional value of $F$. Following the proof of Lemma \ref{L.4}, we ultimately arrive at a contradiction.

Hence all the foregoing discussion shows that $\mathcal{F}$ is normal at $z_0$. So $\mathcal{F}$ is normal in $\Omega$. 
\end{proof}

\begin{lem}\label{LL.1} \cite[Theorem 4.1]{HKR}\cite{NO} Let $f(z)$ be a non-constant entire function such that $\rho(f)\leq 1$ and let $k$ be a positive integer. Then
\[m\left(r,\frac{f^{(k)}}{f}\right)=o(\log r)\;\;(r\to \infty).\]
\end{lem}

\begin{lem} \label{LL.2} \cite[Theorem 2]{LZY1} Let $f(z)$ be a non-constant entire function of finite order, and let $a\neq 0$ be a finite constant. If $f(z)$ and $f^{(k)}(z)$ share the value $a$ CM, then
\[\frac{f^{(k)}(z)-a}{f(z)-a}=c,\]
for some non-zero constant $c$.
\end{lem}

\begin{lem}\label{LL.3} \cite[Theorem 1.2]{LY}
Let $f(z)$ be a non-constant entire function, $k$ be a positive integer and let $Q(z)$ be a polynomial such that all the zeros of $f(z)-Q(z)$ have multiplicity at least $k$. If $f(z)\equiv f^{(k)}(z)$, then one of the following cases must occur:
\begin{enumerate}
\item[(1)] $k=1$ and $f(z)=A\exp(z)$, where $A(\neq 0)$ is constant,

\smallskip
\item[(2)] $k=2$, $Q(z)$ reduces to a constant. If $Q(z)\equiv 0$, then $f(z)=A\exp(\lambda z)$, where $A$ and $\lambda$ are non-zero constants such that $\lambda^2=1$. If $Q(z)\equiv a$, a non-zero constant, then $f(z)=c_0\exp(z)+c_1\exp(-z)$, where $c_0$ and $c_1$ are non-zero constants such that $c_0c_1=\frac{a^2}{4}$,

\smallskip
\item[(3)] $k\geq 3$, $Q(z)$ reduces to $0$ and $f(z)=A\exp(\lambda z)$, where $A$ and $\lambda$ are non-zero constants such that $\lambda^k=1$.
\end{enumerate}
\end{lem}

\section {{\bf Proofs of the main results}} 

\begin{proof}[{\bf Proof of Theorem \ref{t1.1}}]
By the given conditions, we have $(i)$ all the zeros of $f(z)-a$ have multiplicity at least $k$, where $k$ is a positive integer, $(ii)$ $f(z)=a\mapsto f^{(k)}(z)=a$, $(iii)$ $f^{(k)}(z)=a\Rightarrow  f(z)=a$ and $(iv)$ $f^{(k)}(z)=b\Rightarrow f(z)=b$. Since $a\neq 0$ and all the zeros of $f(z)-a$ have multiplicity at least $k$, we can deduce that all the zeros of $f(z)-a$ have multiplicity exactly $k$.

We now claim that $f$ is a transcendental entire function. If not, suppose that $f$ is a non-constant polynomial of degree $m$. Suppose that 
\begin{align*}
f(z)-a=a_m (z-z_1)^{m_1}(z-z_2)^{m_2}\ldots (z-z_l)^{m_l},
\end{align*} 
where $m_1+m_2+\ldots+m_l=m$ and $a_m\not=0$ is a constant. Note that $\deg(f^{(k)})=m-k$. Since $f(z)=a\mapsto f^{(k)}(z)=a$ and $f^{(k)}(z)=a\Rightarrow f(z)=a$, we may assume that 
\begin{align*} 
f^{(k)}(z)-a=b_m (z-z_1)^{n_1}(z-z_2)^{n_2}\ldots (z-z_l)^{n_l},
\end{align*} 
where $n_1+n_2+\ldots+n_l=m-k$ and $b_m\not=0$ is a constant. Again since $f(z)=a\mapsto f^{(k)}(z)=a$, we must have $n_i\geq m_i$ for $i=1,2,\ldots,l$ and so $n_1+n_2+\ldots+n_l\geq m$. Therefore we get a contradiction. Hence $f$ is a transcendental entire function.

Now we prove that $\rho(f)\leq 1$.
Set $\mathcal{F}=\{f_{\omega}\}$, where $f_{\omega}(z)=f(\omega+z)$, $\omega\in\mathbb{C}$.
Then $\mathcal{F}$ is a family of holomorphic functions on the unit disc $\Delta$. Now for each $f_{\omega}\in\mathcal{F}$, we have 
\begin{enumerate}
\item[(1)] all the zeros of $f_{\omega}(z)-a$ have multiplicity at least $k$,
\item[(2)] $f_{\omega}(z)=a\Rightarrow f^{(k)}_{\omega}(z)=a$,
\item[(3)] $f^{(k)}_{\omega}(z)=b\Rightarrow f_{\omega}(z)=b$.
\end{enumerate}
 
Clearly by Lemma \ref{L.4}, we see that $\mathcal{F}$ is normal in $\Delta$ and so by
Lemma \ref{L.1}, there exists $M >0$ satisfying $f^{\#}(\omega)=f^{\#}_{\omega}(0)<M$ for all $\omega\in\mathbb{C}$. Consequently, by Lemma \ref{L.3}, we have $\rho(f)\leq 1$.
We put 
\begin{align}
\label{sm.1} \varphi(z)= \frac{f^{(k+1)}(z)(f^{(k)}(z)-f(z))}{(f(z)-a)(f^{(k)}(z)-b)}.
\end{align}

Now we consider the following two cases.\par

\medskip
{\bf Case 1.} Let $\varphi\equiv 0$. Then either $f(z)\equiv f^{(k)}(z)$ or $f^{(k+1)}(z)\equiv 0$. Since $f(z)$ is a transcendental entire function, we have $f^{(k+1)}(z)\not\equiv 0$. Hence $f(z)\equiv f^{(k)}(z)$. Consequently, by Lemma \ref{LL.3}, we see that one of the following cases must occur:
\begin{enumerate}
\item[(1)] $k=1$ and $f(z)=A\exp(z)$, where $A\in\mathbb{C}\backslash \{0\}$,
\item[(2)] $k=2$ and $f(z)=c_0\exp(z)+c_1\exp(-z)$, where $c_0, c_1\in\mathbb{C}\backslash \{0\}$ such that $c_0c_1=\frac{a^2}{4}$.
\end{enumerate}

\medskip
{\bf Case 2.} Let $\varphi\not\equiv 0$. Using given conditions, one can easily prove that $\varphi$ is an entire function.
Now from (\ref{sm.1}), we have 
\begin{align}
\label{sm.2} \varphi(z)= \frac{f^{(k)}(z)}{f(z)-a}\frac{f^{(k+1)}(z)}{f^{(k)}(z)-b}-\frac{f^{(k+1)}(z)}{f^{(k)}(z)-b}-a\frac{f^{(k)}(z)}{f(z)-a}\frac{f^{(k+1)}(z)}{(f^{(k)}(z)-b)f^{(k)}(z)}.
\end{align}

Therefore using Lemma \ref{LL.1} to (\ref{sm.2}), we get
\begin{align*}
 T(r,\varphi)=m(r,\varphi)&\leq m\left(r,\frac{f^{(k)}}{f-a}\right)+m\left(r,\frac{f^{(k_1)}}{f^{(k)}-b}\right)+m\left(r,\frac{f^{(k+1)}}{f^{(k)}-b}\right)\\&+m\left(r,\frac{f^{(k)}}{f-a}\right)+m\left(r,\frac{f^{(k+1)}}{(f^{(k)}-b)f^{(k)}}\right)+O(1)=o(\log r)\end{align*}
as $r\to\infty$ and so $\varphi$ is a constant. Suppose $\varphi=c_1\in\mathbb{C}\backslash \{0\}$. Now from (\ref{sm.1}), we have
\begin{align}
\label{sm.2} c_1=\frac{f^{(k+1)}(z)(f^{(k)}(z)-f(z))}{(f(z)-a)(f^{(k)}(z)-b)}.
\end{align}

Now we consider the following two sub-cases.\par

\medskip
{\bf Sub-case 2.1.} Suppose $f(z)-a$ has only finitely many zeros. Then by Hadamard's factorization theorem, we can assume that 
\begin{align}
\label{sm.1a} f(z)-a=P(z)\exp(Q(z)),
\end{align}
where $P(z)(\not\equiv 0)$ and $Q(z)$ are polynomials such that $\deg(Q)=1$. Now differentiating $k$-times, we get
\begin{align}
\label{sm.1b} f^{(k)}(z)=\left(P(z)(Q^{(1)}(z))^k+P_{1}(z)\right)\exp(Q(z)),
\end{align}
where $P_1(z)$ is a polynomial such that $\deg(P_1)<\deg\left(P(Q^{(1)})^k\right)$.
Clearly $f^{(k)}(z)-b$ has infinitely many zeros. Let $\{z_n\}$ be a sequence of zeros of $f^{(k)}(z)-b$ such that $P(z_n)\neq 0$ and \begin{align*}
P(z_n)(Q^{(1)}(z_n))^k+P_{1}(z_n)\neq 0
\end{align*}
for all $n\in\mathbb{N}$. Then $f^{(1)}(z_n)=b$. Since $f^{(1)}(z)=b\Rightarrow f(z)=b$, we have $f(z_n)=b$. Therefore from (\ref{sm.1a}) and (\ref{sm.1b}), we have respectively 
\begin{align*}
P(z_n)\exp(Q(z_n))=b-a\;\;\text{and}\;\;\left(P(z_n)(Q^{(1)}(z_n))^k+P_{1}(z_n)\right)\exp(Q(z_n))=b.
\end{align*} 

Eliminating $\exp(Q(z_n))$, we get 
\begin{align*}
\frac{P(z_n)(Q^{(1)}(z_n))^k+P_{1}(z_n)}{P(z_n)}=\frac{b}{b-a}
\end{align*}
for all $n\in\mathbb{N}$, which shows that 
\begin{align*}
\frac{P(z)(Q^{(1)}(z))^k+P_{1}(z)}{P(z)}\equiv \frac{b}{b-a},
\end{align*}
i.e.,
\begin{align}
\label{sm.1c} P(z)\left((Q^{(1)}(z))^k-\frac{b}{b-a}\right)+P_{1}(z)\equiv 0.
\end{align}

Clearly from (\ref{sm.1c}), we get that $P(z)$ is a constant and $(Q^{(1)}(z))^k=\frac{b}{b-a}$. Therefore 
\begin{align*}
f(z)=C\exp\left(\lambda z\right)+a,
\end{align*}
where $C$ is a non-zero constant and $\lambda^k=\frac{b}{b-a}\neq 1$. Since $f^{(k)}(z)=a\Rightarrow  f(z)=a$, we get a contradiction.\par

\medskip
{\bf Sub-case 2.2.} Suppose $f(z)-a$ has infinitely many zeros. We know that all the zeros of $f(z)-a$ have multiplicity exactly $k$. Let $z_0$ be a zero of $f(z)-a$ of multiplicity $k$. Since $f(z)=a\mapsto f^{(k)}(z)=a$ and $f^{(k)}(z)=a\Rightarrow f(z)=a$, we can say that $z_0$ is also a zero of $f^{(k)}(z)-a$ of multiplicity $q_0$, where $q_0\geq k$. Then by simple calculations, one can easily deduce from (\ref{sm.1}) that $q_0=k$. Consequently $f(z)=a\rightleftharpoons f^{(k)}(z)=a$, i.e., $f$ and $f^{(k)}$ share $a$ CM. Then by Lemma \ref{LL.2}, we have
\begin{align}
\label{sm.4} f^{(k)}(z)-a=\lambda (f(z)-a),
\end{align}
where $\lambda $ is a non-zero constant. Clearly $\lambda\not\equiv 1$. If possible, suppose $0$ is not a Picard exceptional value of $f^{(k)}(z)-b$. Let $z_0$ be a zero of $f^{(k)}(z)-b$. By the assumption, we have $f(z_0)=b$. Putting $z=z_0$ into (\ref{sm.4}), we get $b-a=\lambda (b-a)$ and so $\lambda =1$, which is impossible. Hence $0$ is a Picard exceptional value of $f^{(k)}(z)-b$. Then by Hadamard's factorization theorem, we can assume that 
\begin{align}
\label{sm.5} f^{(k)}(z)-b=\exp(Q(z)),
\end{align}
where $Q(z)$ is polynomial such that $\deg(Q)=1$.  Now from (\ref{sm.4}) and (\ref{sm.5}), we get
\begin{align}
\label{sm.6} f(z)=a_1+a_2\exp(Q(z)),
\end{align}
where $a_1=\frac{b-a+a\lambda}{\lambda}$ and $a_2=\frac{1}{\lambda}$. Differentiating (\ref{sm.6}) $k$-times, we get
\begin{align}\label{sm.7} f^{(k)}(z)=a_2(Q^{(1)}(z))^k\exp(Q(z)).
\end{align} 

Now from (\ref{sm.5}) and (\ref{sm.7}), we have
\begin{align*}
 \left(a_2(Q^{(1)}(z))^k-1\right)\exp(Q(z))=b,
 \end{align*}
which is impossible since $b\neq 0$. 
\end{proof}

\begin{proof}[{\bf Proof of Theorem \ref{t1.2}}]
By the given conditions, we have $(i)$ all the zeros of $f(z)-a$ have multiplicity at least $k$, where $k$ is a positive integer, $(ii)$ $f(z)=a\rightleftharpoons f^{(k)}(z)=a$ and $(iii)$ $f^{(k)}(z)=0\Rightarrow f(z)=0$. Obviously all the zeros of $f(z)-a$ have multiplicity exactly $k$. Again the conditions $f(z)=a\rightleftharpoons f^{(k)}(z)=a$ imply that $f(z)$ is a transcendental entire function and so $f^{(k)}(z)\not\equiv 0$. Proceeding in the same way as done in the proof Theorem \ref{t1.1}, we can easily prove that $\rho(f)\leq 1$.

Since $f(z)=a\rightleftharpoons f^{(k)}(z)=a$, by Lemma \ref{LL.2}, we have
\bea\label{smm.1} f^{(k)}(z)-a=\beta (f(z)-a),\eea
where $\beta$ is a non-zero constant.

\medskip
First we assume that $\beta=1$. Then from (\ref{smm.1}), we have $f(z)\equiv f^{(k)}(z)$ and so by Lemma \ref{LL.3}, we see that one of the following cases must occur:
\begin{enumerate}
\item[(1)] $k=1$ and $f(z)=A\exp(z)$, where $A\in\mathbb{C}\backslash \{0\}$,
\item[(2)] $k=2$ and $f(z)=c_0\exp(z)+c_1\exp(-z)$, where $c_0, c_1\in\mathbb{C}\backslash \{0\}$ such that $c_0c_1=\frac{a^2}{4}$.
\end{enumerate}

\medskip
Next we assume that $\beta\neq 1$. Now we consider the following two cases.\par

\medskip
{\bf Case 1.} Suppose $f(z) - a$ has only finitely many zeros. Then by Hadamard's factorization theorem, we may assume that
$f(z)-a=P(z)\exp(Q(z))$, where $P(z)(\not\equiv 0)$ and $Q(z)$ are polynomials such that $\deg(Q)=1$.
Now using (\ref{sm.1a}) and (\ref{sm.1b}) to (\ref{smm.1}), we get
\begin{align*}
\left(P(z)\left((Q^{(1)}(z))^k-\beta\right)+P_1(z)\right)\exp(Q(z))=a.
\end{align*}

Proceeding in the same manner as in the proof of Sub-case 2.1 of Theorem \ref{t1.1}, we obtain a contradiction.

\medskip
{\bf Case 2.} Suppose $f(z)-a$ has infinitely many zeros. If possible, suppose $0$ is not a Picard exceptional value of $f^{(k)}(z)$. Let $z_0$ be a zero of $f^{(k)}(z)$. By the assumption, we have $f(z_0)=0$. Putting $z=z_0$ into (\ref{smm.1}), we get $0-a=\beta (0-a)$ and so $\beta=1$, which is impossible. Hence $0$ is a Picard exceptional value of $f^{(k)}(z)$. Then by Hadamard's factorization theorem, we can assume that 
\begin{align}
\label{smm.2} f^{(k)}(z)=\exp(Q(z)),
\end{align}
where $Q(z)$ is polynomial such that $\deg(Q)=1$. Now from (\ref{smm.1}) and (\ref{smm.2}), we get
\begin{align}
\label{smm.3} f(z)=a_1+a_2\exp(Q(z)),
\end{align}
where $a_1=\frac{-a+a\beta}{\beta}$ and $a_2=\frac{1}{\beta}$. Differentiating (\ref{smm.3}) $k$-times, we get
\begin{align}
\label{smm.4} f^{(k)}(z)=a_2(Q^{(1)}(z))^k\exp(Q(z)).
\end{align}

Now from (\ref{smm.2}) and (\ref{smm.4}), we have $a_2(Q^{(1)}(z))^k=1$, i.e., $(Q^{(1)}(z))^k=\beta$.
Finally, we have
\begin{align*}
f(z)=\frac{a(\beta-1)}{\beta}+\frac{C}{\beta}\exp(\lambda z),
\end{align*}
where $C$, $\beta(\neq 1)$ and $\lambda$ are non-zero constants such that $\lambda^k=\beta$.
\end{proof}

\begin{proof}[{\bf Proof of Theorem \ref{t1.3}}]
By the given conditions, we have $(i)$ either $f(z)-a$ has no zeros or all the zeros of $f(z)-a$ have multiplicity at least $k$, where $k$ is a positive integer, $(ii)$ $f(z)=a\mapsto f^{(k)}(z)=a$ and $(iii)$ $f^{(k)}(z)=b\Rightarrow f(z)=b$.  Now using Lemma \ref{L.5} and proceeding in the same way as done in the proof Theorem \ref{t1.1}, we can easily prove that $\rho(f)\leq 1$.
We consider the following two cases.\par
\smallskip

{\bf Case 1.} Suppose $f(z)-a$ has no zeros. Then by Hadamard's factorization theorem, we assume that 
\begin{align}
\label{k.15} f(z)-a=d_1\exp(c z),
\end{align}
where $c$ and $d_1$ are constants. Since $f(z)$ is non-constant we have $cd_1\neq 0$.
Also from (\ref{k.15}), we get
\begin{align}
\label{k.16} f^{(k)}(z)-b=d_1c^k\exp(c z)-b \;\text{and}\;f(z)-b=d_1\exp(c z)+a-b.
\end{align}

Let $z_2$ be a zero of $f^{(k)}(z)-b$. Since $f^{(k)}(z)=b\Rightarrow f(z)=b$, it follows that $z_2$ is a zero of $f(z)-b$. Then from (\ref{k.16}), we have 
\begin{align}
\label{k.17} d_1c^k\exp(c z_2)=b \;\text{and}\;d_1\exp(c z_2)=b-a.
\end{align}

Now eliminating $\exp(cz_2)$ from (\ref{k.17}), we get $c^k=\frac{b}{b-a}$. Since $a\neq b$, we have $c^k\neq 1$. Therefore
\begin{align*}
f(z)=d_1\exp(c z)+a,
\end{align*}
where $c, d_1\in\mathbb{C}\backslash \{0\}$ such that $c^k=\frac{b}{b-a}\neq 1$.\par
\smallskip

{\bf Case 2.} All the zeros of $f(z)-a$ have multiplicity at least $k$. Since $a\neq 0$ and all the zeros of $f(z)-a$ have multiplicity at least $k$, we can deduce that all the zeros of $f(z)-a$ have multiplicity exactly $k$.
Now we prove that $f^{(k)}(z)\not\equiv b$. If not, suppose that $f^{(k)}(z)\equiv b$. On integration, we have 
\begin{align*}
f(z)=bz^k+b_{k-1}z^{k-1}+\ldots+b_0,
\end{align*}
where $b_i\in\mathbb{C}$ for $i=0,1,\ldots,k-1$. Let $z_1$ be a zero of $f(z)-a$. Since $f(z)=a\mapsto f^{(k)}(z)=a$, it follows that $f^{(k)}(z_1)=a$. On the other hand we have $f^{(k)}(z_1)=b$ and so $a=b$, which is impossible. Hence $f^{(k)}(z)\not\equiv b$.
We assume the same auxiliary function $\varphi$ defined as in (\ref{sm.1}).

Now we consider the following two sub-cases.\par

\medskip
{\bf Case 2.1.} Let $\varphi\equiv 0$. Then either $f(z)\equiv f^{(k)}(z)$ or $f^{(k+1)}(z)\equiv 0$. Suppose that $f^{(k+1)}(z)\equiv 0$. On integration, we have 
\begin{align*}
f(z)=a_kz^k+a_{k-1}z^{k-1}+\ldots+a_0,
\end{align*}
where $a_k,\ldots, a_0\in\mathbb{C}$ and at least one of $a_k,\ldots, a_1\in\mathbb{C}$ is non-zero. Since all the zeros of $f-a$ have multiplicity at least $k$, we assume that 
\[f(z)-a=a_k(z-z_1)^k,\]
where $z_1\in\mathbb{C}$. Note that $f^{(k)}(z)=k! a_k$. Since $f(z)=a\mapsto f^{(k)}(z)=a$, it follows that $a=k! a_k$. Therefore 
\begin{align*}
f(z)=\frac{a}{k!}(z-z_1)^k+a.
\end{align*}

Consequently, by Lemma \ref{LL.3}, we see that one of the following cases must occur:
\begin{enumerate}
\item[(1)] $k=1$ and $f(z)=A\exp(z)$, where $A\in\mathbb{C}\backslash \{0\}$,
\item[(2)] $k=2$ and $f(z)=c_0\exp(z)+c_1\exp(-z)$, where $c_0, c_1\in\mathbb{C}\backslash \{0\}$ such that $c_0c_1=\frac{a^2}{4}$,
\item[(3)] $f(z)=\frac{a}{k!}(z-z_1)^k+a$.
\end{enumerate}

\medskip
{\bf Case 2.2.} Let $\varphi\not\equiv 0$. Then $f(z)\not \equiv f^{(k)}(z)$ and $f^{(k+1)}(z)\not\equiv 0$. Also by the given conditions, one can easily prove that $\varphi$ is an entire function. Now proceeding in the same way as done in the proof of Case 2 of Theorem \ref{t1.1}, we can prove that $\varphi$ is a non-zero constant and so we arrive at (\ref{sm.2}).

\medskip
Denote by $N(r,0; f\mid g\not=a_0)$ the counting function of those zeros of $f$ which are not the $a_0$-points of $g$.

Immediately from (\ref{sm.2}), we have 
\begin{align}
\label{k.4} N\big(r,0;f^{(k+1)}\mid f^{(k)}\neq b\big)=0.
\end{align}

We consider the following two sub-cases.\par

\medskip
{\bf Sub-case 2.2.1.} Let $k=1$. By the given conditions, we have $(i)$ $f(z)=a\Rightarrow f^{(1)}(z)=a$ and $(ii)$ $f^{(1)}(z)=b\Rightarrow f(z)=b$. Consequently, by Theorem 1.E, we get the required conclusions.\par

\medskip
{\bf Sub-case 2.2.2.} Let $k\geq 2$. Let $z_0$ be a zero of $f(z)-a$ of multiplicity $k\geq 2$. Since $f(z)=a\mapsto f^{(k)}(z)=a$, $z_0$ is also a zero of $f^{(k)}(z)-a$ of multiplicity $q_0\geq k\geq 2$. Clearly $z_0$ is a zero of $f^{(k+1)}(z)$ of multiplicity $q_0-1\geq 1$. On the other hand from (\ref{k.4}), we have $N\left(r,0;f^{(k+1)}\mid f^{(k)}\neq b\right)=0$. Therefore we get a contradiction. Hence $a$ is a Picard exceptional value of $f$, which is impossible.
\end{proof}

\begin{proof}[{\bf Proof of Theorem \ref{t1.4}}]
By the given conditions, we have $(i)$ all the zeros of $f(z)-a$ have multiplicity at least $k$, where $k$ is a positive integer, $(ii)$ $f(z)=a\Rightarrow f^{(k)}(z)=a$ and $(iii)$ $f^{(k)}(z)=0\mapsto f(z)=0$. Since $a\neq 0$ and all the zeros of $f(z)-a$ have multiplicity at least $k$, we get that all the zeros of $f(z)-a$ have multiplicity exactly $k$. Also by the given condition, we have $\lambda(f)<\rho(f)$ and so $f(z)$ is a transcendental entire function. Clearly $f^{(k)}(z)\not\equiv 0$. Now using Lemma \ref{L.6} and then proceeding in the same way as done in the proof Theorem \ref{t1.1}, we can easily prove that $\rho(f)\leq 1$.
Therefore, we have $0\leq \lambda(f)<\rho(f)\leq 1$. If $0<\rho(f)<1$, then by Lemma 4.10.1 \cite{Hol}, we get $\rho(f)=\lambda(f)$, which is impossible. Hence $\rho(f)=1$. Again since $f^{(k)}(z)=0\mapsto f(z)=0$, we get
\begin{align}
\label{k2.1} N(r,0;f^{(k)}\mid f\neq 0)=0.
\end{align}

Let 
\begin{align}
\label{k2.2} H(z)=\frac{f(z)}{f^{(k)}(z)}.
\end{align}

Let $z_1$ be a zero of $f^{(k)}(z)$ of multiplicity $q_1$. Since $f^{(k)}(z)=0\mapsto f(z)=0$, $z_1$ is also a zero of $f(z)$ of multiplicity $p_1$. Clearly $q_1\leq p_1$ and so from (\ref{k2.2}), we see that $z_1$ is a zero of $H(z)$ of multiplicity $p_1-q_1$. Therefore, it is clear that
\begin{align}
\label{k2.3} N(r,0;H)\leq N(r,0;f).
\end{align}

Clearly from (\ref{k2.1}) and (\ref{k2.2}), one can easily conclude that $H(z)$ is an entire function.

We consider the following two cases.\par

\medskip
{\bf Case 1.} Let $H(z)$ be a constant. Clearly $H(z)$ is a non-zero constant. Then from (\ref{k2.3}), we get $\ol N(r,0;f)=0$. If $\ol N(r,a;f)=0$, then by second fundamental theorem, we get a contradiction. Hence $\ol N(r,a;f)\neq 0$. Let $z_1$ be a zero of $f(z)-a$. Since $f(z)=a\Rightarrow f^{(k)}(z)=a$, we have $f^{(k)}(z_1)=a$ and so from (\ref{k2.2}), we get $H(z)=1$, which implies that 
\begin{align}
\label{k2.3a} f(z)\equiv f^{(k)}(z).
\end{align}

Since $\lambda(f)<\rho(f)=1$, we conclude that $0$ and $\infty$ are the Borel exceptional values of $f(z)$. Then by Theorem 2.12 \cite{10}, we have 
\begin{align}
\label{k2.3b} \delta(0;f)=\delta(\infty;f)=1.
\end{align}

Now we consider the  following two sub-cases.\par

\medskip
{\bf Sub-case 1.1.} Let $k=1$. Now using Lemma \ref{LL.3} to (\ref{k2.3a}), we have $f(z)=A\exp(z)$, where $A$ is a non-zero constant.\par

\medskip 
{\bf Sub-case 1.2.} Let $k\geq 2$. Now using Lemma \ref{LL.3} to (\ref{k2.3a}), we have
\begin{align}
\label{k2.3c}f(z)=c_0\exp(z)+c_1\exp(-z)=\frac{c_0\exp(2z)+c_1}{\exp(z)},
\end{align}
where $c_0, c_1\in\mathbb{C}\backslash \{0\}$ such that $c_0c_1=a^2/4$.
Clearly from (\ref{k2.3c}), we have 
\begin{align*}
T(r,f)=2T(r,\exp(z))+S(r,\exp(z)).
\end{align*}

Since $\exp(z)\neq 0,\infty$, we get 
\begin{align*}
N(r,\omega;\exp(z))\sim T(r,\exp(z)),
\end{align*}
where $\omega\in S=\{z: c_0z^2+c_1=0\}$. Therefore
\begin{align*} 
\delta(0;f)= 1- \limsup\limits_{r\lra \infty}\frac{N(r,0;f)}{T(r,f)}
&=1- \limsup\limits_{r\lra \infty}\frac{\sum\limits_{\omega\in S}N(r,\omega;\exp(z))}{2T(r,\exp(z))+S(r,\exp(z))}
\\&=1- \limsup\limits_{r\lra \infty}\frac{2T(r,\exp(z))}{2T(r,\exp(z))+S(r,\exp(z))}=0,
\end{align*}
which contradicts (\ref{k2.3b}).

\medskip
{\bf Case 2.} Let $H(z)$ be non-constant.
We now consider the following two sub-cases.\par

\medskip
{\bf Sub-case 2.1.} Suppose $f(z)$ has only finitely many zeros. Then by Hadamard's factorization theorem, we assume that $f(z)=P(z)\exp(cz)$,
where $P(z)$ is a non-zero polynomial and $c$ is a non-zero constant. Note that $f^{(k)}(z)=P_k(z)\exp(cz)$, where $P_k(z)$ is a non-zero polynomial. Then from (\ref{k2.2}), we see that $H(z)$ is a non-constant polynomial. Note that 
\begin{align*}
H(z)-1=\frac{f(z)-f^{(k)}(z)}{f^{(k)}(z)}.
\end{align*}

Since $f(z)=a\Rightarrow f^{(k)}(z)=a$, it follows that 
$\ol N(r,a;f)\leq N(r,0;H-1)\leq T(r,H)=O(\log r)$. 
Therefore by the second fundamental theorem, we have
\begin{align*}
T(r,f)\leq \ol N(r,f)+\ol N(r,0;f)+\ol N(r,a;f)+S(r,f)=O(\log r)+S(r,f)=S(r,f),
\end{align*}
which is impossible.\par

\medskip
{\bf Sub-case 2.2.} Suppose $f(z)$ has infinitely many zeros. Note that $\lambda(f)<\rho(f)$.
This shows that $0$ is a Borel exceptional value of $f(z)$. Note that $\infty$ is a Picard exceptional value of $f(z)$. Since $\rho(f)>0$, it follows that $\infty$ is also a Borel exceptional value of $f(z)$. Then by Theorem 2.11 \cite{10}, we get $\mu(f)=\rho(f)=1$. Let
\begin{align}
\label{k2.4} f(z)=\mathscr{G}(z)\exp(cz),
\end{align}
where $\mathscr{G}(z)$ is the canonical products formed with the zeros of $f(z)$.
By Ash \cite[Theorem 4.3.6]{RBA}, we have 
\begin{align}
\label{k2.5}\lambda\left(\mathscr{G}\right)=\rho\left(\mathscr{G}\right)=\lambda(f).
\end{align}

Now from (\ref{k2.3}) and (\ref{k2.4}), we get $N(r,0;H)\leq N(r,0;f)=N(r,0;\mathscr{G})$.
Since the exponent of convergence of the zeros of an integral function is not increased by removing some of the terms, we have $\rho(H)=\lambda(H)\leq \lambda(\mathscr{G})$. Since $\lambda(f)<\rho(f)$, from (\ref{k2.5}), we get $\rho(H)<\rho(f)$. Again since $\mu(f)=\rho(f)$, it follows that $\rho(H)<\mu(f)=1$. Then by Theorem 1.18 \cite{10}, we have $T(r,H)=o(T(r,f))\;(r\to \infty )$,  i.e., $T(r,H)=S(r,f)$, i.e., $H(z)$ is a small function of $f(z)$.
Now (\ref{k2.3}) yields 
\begin{align*}
\ol N(r,0;f)\leq N(r,0;H)\leq T(r,H)=S(r,f).
\end{align*}

Note that $H(z)-1=\frac{f(z)-f^{(k)}(z)}{f^{(k)}(z)}$.
Since $f(z)=a\Rightarrow f^{(k)}(z)=a$, we get 
\begin{align*}
\ol N(r,a;f)\leq N(r,0;H-1)\leq T(r,H)=S(r,f).
\end{align*}

Since $\ol N(r,0;f)+\ol N(r,a;f)=S(r,f)$, by the second fundamental theorem, we get a contradiction.
\end{proof}

\begin{proof}[{\bf Proof of Theorem \ref{t1.5}}]
It is clear that $f(z)$ is a transcendental entire function and so $f^{(k)}(z)\not\equiv 0$.
Now we prove that $\rho(f)=1$. Set $\mathcal{F}=\{f_{\omega}\}$, where $f_{\omega}(z)=f(\omega+z)$, $\omega\in\mathbb{C}$.
Then $\mathcal{F}$ is a family of holomorphic functions on the unit disc $\Delta$. Note that for each $f_{\omega}\in\mathcal{F}$, we have 
\begin{enumerate}
\item[(1)] either $f(z)$ has no zeros or all the zeros of $f_{\omega}(z)$ have multiplicity at least $k$,
\item[(2)] $f_{\omega}(z)=a\Rightarrow f^{(k)}_{\omega}(z)=a$, 
\item[(3)] $f^{(k)}_{\omega}(z)=0\Rightarrow f_{\omega}(z)=0$. 
\end{enumerate}

\medskip
Then $\mathcal{F}$ is normal in $\Delta$ if we use Lemma \ref{L.7}. Now by Lemma \ref{L.1}, we have $f^{\#}(\omega)<M$ for all $\omega\in\mathbb{C}$, where $M > 0$. Hence using Lemma \ref{L.3}, we get $\rho(f)\leq 1$. Also it is easy to prove that $\rho(f)=1$. Since $f^{(k)}(z)=0\Rightarrow f(z)=0$, we have $N(r,0;f^{(k)}\mid f\neq 0)=0$. We consider the auxiliary function $H(z)$ defined as in (\ref{k2.2}).

Let $z_1$ be a zero of $f^{(k)}(z)$ of multiplicity $q_1$. Since $f^{(k)}(z)=0\Rightarrow f(z)=0$, we see that $z_1$ is also a zero of $f(z)$ of multiplicity $p_1$. Clearly $p_1\geq k$ as the zeros of $f(z)$ have multiplicity at least $k$.
Since $f^{(k)}(z_1)=0$ and $z_1$ is a zero of $f(z)$ of multiplicity $p_1(\geq k)$, it follows that $p_1\geq k+1$. Therefore from (\ref{k2.2}), we get that $z_1$ is a zero of $H(z)$ of multiplicity $k$ and so all zeros of $H(z)$ have multiplicity exactly $k$. It is clear that $f(z)=0\Leftrightarrow H(z)=0$ and so $N(r,0;H)\leq N(r,0;f)$.
Hence from (\ref{k2.1}) and (\ref{k2.2}), we conclude that $H(z)$ is an entire function.
We consider the following two cases.\par

\medskip
{\bf Case 1.} Let $H(z)$ be a constant. Then from (\ref{k2.3}), we get $\ol N(r,0;f)=0$. Obviously $\ol N(r,a;f)\neq 0$. Let $z_1$ be a zero of $f(z)-a$. Since $f(z)=a\Rightarrow f^{(k)}(z)=a$, we have $f^{(k)}(z_1)=a$ and so from (\ref{k2.2}), we get $H(z)=1$, which implies that 
\bea\label{kk2.3a} f(z)\equiv f^{(k)}(z).\eea 

Now using Lemma \ref{LL.3} to (\ref{kk2.3a}), we have $f(z)=A\exp(\lambda z)$, where $A$ and $\lambda$ are non-zero constants
 such that $\lambda^k=1$.\par

\medskip
{\bf Case 2.} Let $H(z)$ be non-constant. Now proceeding in the same way as done in the proof of Case 2 of Theorem \ref{t1.4}, we get a contradiction.
\end{proof}

\section {{\bf Some applications}} 
Solving nonlinear differential equations remains one of the most challenging and intellectually stimulating problems in mathematical analysis. In this section, we consider the following non-linear differential equation
\bea\label{all.1} f^{(k+1)}(z)\left(f^{(k)}(z)-f(z)\right)=\varphi(z)(f(z)-a)\left(f^{(k)}(z)-b\right).\eea
where $\varphi(z)$ is an entire function, $a(\neq 0)$ and $b$ are two finite complex numbers.

The general solvability of (\ref{all.1}) is extremely delicate. The equation is nonlinear, involves mixed products of derivatives, and includes an arbitrary entire coefficient function $\varphi(z)$. With no growth restrictions on $\varphi$, it is usually very difficult-even impossible-to explicitly determine whether non-constant entire solutions exist.

Theorems \ref{t1.1}-\ref{t1.5} indicate a deep relationship between the value-sharing behavior of an entire function $f$ and its 
$k$-th derivative $f^{(k)}$ and the analytic structure of the nonlinear differential Eq. (\ref{all.1}). In particular, the assumption that $f$ and $f^{(k)}$ share two distinct values, $a$ and $b$, forces strong functional constraints that can be encoded in equation (\ref{all.1}). In this sense, (\ref{all.1}) serves as a bridge between Nevanlinna-type value-distribution properties and differential equations of higher order.
Equation (\ref{all.1}) is in general highly non-linear and difficult to analyze for arbitrary entire $\varphi$. However, in certain special cases, Eq. (\ref{all.1}) can be solved completely. Now we state our results. 

\begin{theo}\label{tt1} Let $\varphi(z)$ be an entire function, $k$ be a positive integer and let $a$ and $b$ be two non-zero constants such that $a\neq b$. Let $f(z)$ be a meromorphic solution of the differential equation (\ref{all.1}) such that either $f(z)-a$ has no zeros or all the zeros of $f(z)-a$ have multiplicity at least $k$ and $f(z)=a\mapsto f^{(k)}(z)=a$. 
Then the conclusions of Theorem \ref{t1.3} hold.
\end{theo}

\begin{theo}\label{tt2} Let $\varphi(z)$ be an entire function, $k$ be a positive integer and let $a$ be a non-zero constant. Let $f$ be a meromorphic solution of the following equation 
\bea\label{all.2} f^{(k+1)}(z)\left(f^{(k)}(z)-f(z)\right)=\varphi(z) (f(z)-a)f^{(k)}(z)\eea
such that either $f(z)$ has no zeros or all the zeros of $f(z)$ have multiplicity at least $k$, $f(z)=a\Rightarrow f^{(k)}(z)=a$ and $\lambda(f)<\rho(f)$. Then $f(z)=A\exp(\lambda z)$, where $A$ and $\lambda$ are non-zero constants such that $\lambda^k=1$.
\end{theo}

\begin{proof}[{\bf Proof of Theorem \ref{tt1}}] 
Let $f(z)$ be a meromorphic solution of the non-linear differential equation (\ref{all.1}) such that either $f(z)-a$ has no zeros or all the zeros of $f(z)-a$ have multiplicity at least $k$ and $f(z)=a\mapsto f^{(k)}(z)=a$.
Now we consider the following two cases.\par

\medskip
{\bf Case 1.} Let $\varphi\not\equiv 0$. Then $f(z)\not\equiv f^{(k)}(z)$ and $f^{(k+1)}(z)\not\equiv 0$. 
Since $\varphi(z)$ is an entire function, from (\ref{all.1}), it is clear that $f(z)$ is a non-constant entire function. 

Now we prove that $f^{(k)}(z)=b\Rightarrow f(z)=b$. If $b$ is a Picard exceptional value of $f^{(k)}(z)$, then obviously $f^{(k)}(z)=b\Rightarrow f(z)=b$. Next we suppose that $b$ is not a Picard exceptional value of $f^{(k)}(z)$.
Let $z_1$ be a zero of $f^{(k)}(z)-b$ of multiplicity $p_1$. Then from (\ref{all.1}), we see that $z_1$ must be a zero of $f^{(k)}(z)-f(z)$. Since 
\begin{align*}
f^{(k)}(z)-f(z)=f^{(k)}(z)-b-(f(z)-b),
\end{align*}
it follows that $z_1$ is a zero of $f(z)-b$. So $f^{(k)}(z)=b\Rightarrow  f(z)=b$. Now from the proof of Theorem \ref{t1.1}, we can prove that $\rho(f)\leq 1$. Therefore using Lemma \ref{LL.1} to (\ref{all.1}), we obtain that $\varphi$ is a constant. So, we arrive at (\ref{sm.2}). We consider the following two sub-cases.\par

\medskip
{\bf Sub-case 1.1.} Let $k=1$. By the given conditions, we have $(i)$ $f(z)=a\Rightarrow f^{(1)}(z)=a$ and $(ii)$ $f^{(1)}(z)=b\Rightarrow f(z)=b$. Consequently by Theorem 1.E, we get the required conclusions.\par

\medskip
{\bf Sub-case 1.2.} Let $k\geq 2$. Note that $f(z)=a\mapsto f^{(k)}(z)=a$. If all the zeros of $f(z)-a$ have multiplicity at least $k$, then obviously all the zeros of $f(z)-a$ have multiplicity exactly $k$.
Now we consider the following two sub-cases.\par

\medskip
{\bf Sub-case 1.2.1.} Suppose $a$ is not a Picard exceptional value of $f(z)$. Let $z_0$ be a zero of $f(z)-a$ of multiplicity $k\geq 2$. Since $f(z)=a\mapsto f^{(k)}=a$, $z_0$ is also a zero of $f^{(k)}(z)-a$ of multiplicity $q_0$. First we suppose that $q_0\geq k\geq 2$. Clearly $z_0$ is a zero of $f^{(k)}(z)-f(z)$ and $f^{(k+1)}(z)$ of multiplicities $k$ (at least) and $q_0-1\geq 1$ respectively. Therefore from (\ref{sm.2}), we get a contradiction. Hence $q_0<k$ and so from (\ref{sm.2}), we get $f^{(k+1)}(z_0)=0$. Consequently 
\begin{align}
\label{kkk.1}f(z)=a\Rightarrow f^{(k+1)}(z)=0.
\end{align}

Since all the zeros of $f(z)-a$ have multiplicity exactly $k$, we may take 
\begin{align}
\label{kkk.1a} f(z)-a=g^k(z),
\end{align}
where $g(z)$ is a non-constant entire function having only simple zeros. Now from the proof of Theorem 1.1 \cite{Mss}, we can find that
\begin{align}
\label{kkk.2} f^{(k)}(z)
&= k!(g^{(1)}(z))^{k}+\frac{k(k-1)}{2}k!g(z)(g^{(1)}(z))^{k-2}g^{(2)}(z)+\ldots +kg^{k-1}(z)g^{(k)}(z)\\&=
k!(g^{(1)}(z))^{k}+\frac{k(k-1)}{2}k!g(z)(g^{(1)}(z))^{k-2}g^{(2)}(z)+ R_{1}(g(z)),\nonumber
\end{align}
where $R_{1}(g(z))$ is a differential polynomial in $g(z)$ with constant coefficients and each term of $R_{1}(g(z))$ contains $g^{m}(z) (2 \leq m \leq k-1$) as a factor.

Differentiating (\ref{kkk.2}) once, we get 
\begin{align}
\label{kkk.3} f^{(k+1)}(z)= \frac{k(k+1)}{2}k!(g^{(1)}(z))^{k-1}g^{(2)}(z)+S_{1}(g(z)),
\end{align} 
where $S_{1}(g(z))$ is a differential polynomial in $g(z)$ and each term of $S_{1}(g(z))$ contains $g(z)$ and its higher powers as a factor.

Let $z_0$ be a zero of $f(z)-a$. Then from (\ref{kkk.1}), we get $f^{(k+1)}(z_0)=0$. Also from from (\ref{kkk.1a}), we see that $z_0$ is a simple zero of $g(z)$, i.e., $g(z_0)=0$, but $g^{(1)}(z_0)\neq 0$. Therefore from (\ref{kkk.3}), we have $g^{(2)}(z_0)=0$ and so $g(z)=0\Rightarrow g^{(2)}(z)=0$.
Let 
\begin{align}
\label{kkk.4} H_1(z)=\frac{g^{(2)}(z)}{g(z)}.
\end{align}

If possible, suppose that $H_1(z)\equiv 0$. Then from (\ref{kkk.4}), we have $g^{(2)}(z)\equiv 0$ and so $g(z)=a_1z+a_0$, where $a_1(\neq 0), a_0\in\mathbb{C}$. Clearly (\ref{kkk.1a}) gives $f(z)-a=(a_1z+a_0)^k$ and so $f^{(k+1)}(z)\equiv 0$, which is impossible. Hence $H_1(z)\not\equiv 0$. Obviously $H_1(z)$ is an entire function. Now using Lemma \ref{LL.1} to (\ref{kkk.4}), we obtain that $H_1(z)$ is a non-zero constant. Let $H_1(z)=\tilde{\lambda}$.
Then from (\ref{kkk.4}) we have $g^{(2)}(z)=\lambda g(z)$. Solving, we get
\begin{align}
\label{kkk.5} g(z)=A_1\exp\left(\sqrt{\lambda}z\right)+B_1\exp\left(-\sqrt{\lambda}z\right),
\end{align}
where $A_1$ and $B_1$ are non-zero constants. Note that 
\begin{align}
\label{kkk.6} g^k(z)=A_1^{k}\exp\left(k\sqrt{\lambda}z\right)+\ldots+B_1^{k}\exp\left(-k\sqrt{\lambda}z\right).
\end{align}

Therefore
\begin{align}
\label{kkk.7} \left(g^k(z)\right)^{(k)}=A_1^{k}\left(k\sqrt{\lambda}\right)^{k}\exp\left(k\sqrt{\lambda}z\right)+\ldots
+B_1^{k}(-1)^k\left(k\sqrt{\lambda}\right)^{k}\exp\left(-k\sqrt{\lambda}z\right)
\end{align}
and 
\begin{align}
\label{kkk.8} \left(g^k(z)\right)^{(k+1)}=A_1^{k}\left(k\sqrt{\lambda}\right)^{k+1}\exp\left(k\sqrt{\lambda}z\right)+\ldots
+B_1^{k}(-1)^{k+1}\left(k\sqrt{\lambda}\right)^{k+1}\exp\left(-k\sqrt{\lambda}z\right).
\end{align}

Now using (\ref{kkk.1a}), (\ref{kkk.6}), (\ref{kkk.7}) and (\ref{kkk.8}) to (\ref{sm.2}), we get
\begin{align}
\label{kkk.9}&\left(A_1^{k}\left(k\sqrt{\lambda}\right)^{k+1}\exp\left(k\sqrt{\lambda}z\right)+\ldots
+B_1^{k}(-1)^{k+1}\left(k\sqrt{\lambda}\right)^{k+1}\exp\left(-k\sqrt{\lambda}z\right)\right)\times\nonumber\\
&\bigg(A_1^{k}\left(k\sqrt{\lambda}\right)^{k}\exp\left(k\sqrt{\lambda}z\right)+\ldots
+B_1^{k}(-1)^k\left(k\sqrt{\lambda}\right)^{k}\exp\left(-k\sqrt{\lambda}z\right)\nonumber\\
&-\left(A_1^{k}\exp\left(k\sqrt{\lambda}z\right)+\ldots+B_1^{k}\exp\left(-k\sqrt{\lambda}z\right)\right)-a\bigg)\nonumber\\
&=c_1\left(A_1^{k}\exp\left(k\sqrt{\lambda}z\right)+\ldots+B_1^{k}\exp\left(-k\sqrt{\lambda}z\right)\right)\times\nonumber\\
&\left(A_1^{k}\left(k\sqrt{\lambda}\right)^{k}\exp\left(k\sqrt{\lambda}z\right)+\ldots
+B_1^{k}(-1)^k\left(k\sqrt{\lambda}\right)^{k}\exp\left(-k\sqrt{\lambda}z\right)-b\right).
\end{align}

Then from (\ref{kkk.9}), we get 
\begin{align}
\label{kkk.10} k\sqrt{\lambda}\left(\left(k\sqrt{\lambda}\right)^k-1\right)=c_1\;\;\text{and}\;\;
-k\sqrt{\lambda}\left(\left(-k\sqrt{\lambda}\right)^k-1\right)= c_1.
\end{align}

Consequently, we have 
\begin{align}
\label{kkk.11}\left(k\sqrt{\lambda}\right)^k-1 = (-1)^{k+1}\left(k\sqrt{\lambda}\right)^k+1.
\end{align}

If $k$ is odd, then from (\ref{kkk.11}), we get a contradiction. If $k$ is even, then from (\ref{kkk.10}) and (\ref{kkk.11}), we have $c_1=0$, which is impossible.\par

\medskip
{\bf Sub-case 1.2.2.} Suppose $a$ is a Picard exceptional value of $f(z)$. Then proceeding in the same way as done in Sub-case 2.2 in the proof of Theorem \ref{t1.2}, we get $f(z)=d\exp(c z)+a$,
where $c$ and $d$ are non-zero constants such that $c^k=\frac{b}{b-a}\neq 1$.\par

\medskip
{\bf Case 2.} Let $\varphi\equiv 0$. In this case, we follow Case 1 of the proof of Theorem \ref{t1.2} to get the required conclusions.
\end{proof}

\begin{proof}[{\bf Proof of Theorem \ref{tt2}}] 
Let $f(z)$ be a meromorphic solution of the equation (\ref{all.2}) such that all the zeros of $f(z)$ have multiplicity at least $k$, $f(z)=a\Rightarrow f^{(k)}(z)=a$ and $\lambda(f)<\rho(f)$. Clearly $f(z)$ is a transcendental meromorphic function and so $f^{(k)}(z)\not\equiv 0$.

Now we consider the following two cases.\par

\medskip
{\bf Case 1.} Let $\varphi\not\equiv 0$. Then $f(z)\not\equiv f^{(k)}(z)$. Also from (\ref{all.2}), we can conclude that $f(z)$ is a non-constant entire function. Now from the proof of Theorem \ref{t1.5}, we can prove that $\rho(f)\leq 1$.
We prove that $f^{(k)}(z)=0\Rightarrow f(z)=0$. 
If $0$ is a Picard exceptional value of $f^{(k)}(z)$, then $f^{(k)}(z)=0\Rightarrow f(z)=0$. Next let $z_1$ be a zero of $f^{(k)}(z)$ of multiplicity $q_1$. Then from (\ref{all.2}), we see that $z_1$ is a zero of $f^{(k)}(z)-f(z)$ and so
$z_1$ is a zero of $f(z)$. Therefore $f^{(k)}(z)=0\Rightarrow  f(z)=0$. Suppose $z_1$ is also a zero of $f(z)$ of multiplicity $p_1$. Clearly $p_1\geq k$ as the zeros of $f(z)$ have multiplicity at least $k$.
Since $f^{(k)}(z_1)=0$ and $z_1$ is a zero of $f(z)$ of multiplicity $p_1(\geq k)$, it follows that $p_1\geq k+1$. Therefore from (\ref{k2.2}), we deduce that $z_1$ is a zero of $H(z)$ of multiplicity $k$ and so all zeros of $H(z)$ have multiplicity exactly $k$. It is clear that $f(z)=0\Leftrightarrow H(z)=0$ and so $N(r,0;H)\leq N(r,0;f).$

Consequently, from (\ref{k2.1}) and (\ref{k2.2}), we conclude that $H(z)$ is an entire function. If $H(z)$ is a constant, then from the proof of Case 1 of Theorem \ref{t1.5}, we get $f(z)\equiv f^{(k)}(z)$, which is impossible. On the other hand if $H(z)$ is non-constant, then proceeding in the same way as done in the proof of Case 2 of Theorem \ref{t1.4}, we get a contradiction.\par

\medskip
{\bf Case 2.} Let $\varphi\equiv 0$. Since $f(z)$ is transcendental, we have $f^{(k+1)}(z)\not\equiv 0$ and so $f(z)\equiv f^{(k)}(z)$.
Then by Lemma \ref{LL.3}, we get $f(z)=A\exp(\lambda z)$,
where $A$ and $\lambda$ are non-zero constants such that $\lambda^k=1$.
\end{proof}

\vspace{0.1in}
{\bf Compliance of Ethical Standards:}\par

{\bf Conflict of Interest.} The authors declare that there is no conflict of interest regarding the publication of this paper.\par

{\bf Data availability statement.} Data sharing not applicable to this article as no data sets were generated or analysed during the current study.

\end{document}